\RequirePackage{amsthm}

\documentclass[pdflatex,sn-mathphys]{article}

\usepackage{authblk}

\usepackage{ytableau}
\usepackage[all]{xy}

\usepackage{color}

\definecolor{darkgreen}{rgb}{0,0.5,0}

\definecolor{dbrown}{RGB}{101,67,33}

\input prepictex   \input pictex    
\input postpictex

\newcounter{boxsize}
\newcounter{tempcounter}
\newcommand\smbox{\put(0,0){\line(1,0){\value{boxsize}}}%
  \put(\value{boxsize},0){\line(0,1){\value{boxsize}}}%
  \put(0,0){\line(0,1){\value{boxsize}}}%
  \put(0,\value{boxsize}){\line(1,0){\value{boxsize}}}}

\newcommand\singlebox[1]{\raisebox{-.4ex}{\begin{picture}(4,0)\setcounter{boxsize}{3}%
    \put(0,0)\smbox%
    \put(0,0){\makebox(\value{boxsize},\value{boxsize})[c]{%
      $\scriptstyle\sf#1$}}\end{picture}}}

\newcommand\boxes[2]{\ifthenelse{#2=3}{$\scriptstyle P_2^{#1}$}{%
                                       $\scriptstyle P_{#2}^{#1}$}}

\newcommand{\al}{\alpha}
\newcommand{\be}{\beta}
\newcommand{\ga}{\gamma}

\newcommand\boxleq{\leq_{\rm box}}
\newcommand\boxl{<_{\rm box}}
\newcommand\domleq{\leq_{\rm dom}}
\newcommand\natleq{\leq_{\rm nat}}

\usepackage{amssymb}
\usepackage{amsmath}
\usepackage{amsthm}

\newtheorem{theorem}{Theorem}
\newtheorem{ex}{Example}%
\newtheorem{rem}{Remark}%
\newtheorem{defin}{Definition}%
\numberwithin{equation}{section}
\newtheorem{lemma}[equation]{\sc Lemma}
\newtheorem{cor}[equation]{\sc Corollary}

\usepackage[figuresright]{rotating}

\begin{document}

\title{Box Order and Dominance Order on Standard Young Tableaux}

\author[1]{Mariusz Kaniecki}
\author[1]{Justyna Kosakowska}

\affil[1]{Faculty of Mathematics
              and Computer Science, Nicolaus Copernicus University in Toru\'n,  ul. Chopina 12/18, Toru\'n, 87-100, Poland}

\date{}

\maketitle

\begin{abstract}
We generalize results of \cite{kst2019} concerning two partial orders on Standard Young Tableaux: the dominance order and the box order. Extending their definitions from tableaux of a~fixed shape to the set $\mathcal{T}_r$ of all Standard Young Tableaux of a fixed size $r$, we prove that these orders coincide in this more general setting. We also investigate basic properties of these orders. We further show that this result extends to Littlewood–Richardson tableaux whose skew shapes are rook strips. We briefly mention applications to invariant subspaces of nilpotent linear operators.
\end{abstract}

{\it Keywords}: 
Standard Young tableaux, partial orders, dominance order, box order, Littlewood-Richardson tableaux, invariant subspaces

{\it MSC2020}: 05E10, 
06A07 

\section{Introduction}

Standard Young Tableaux (SYTs) are one of the important combinatorial tools in algebraic combinatorics. They play a~key role in representation theory and algebraic geometry. SYTs carry natural partial order structures,
and several such orders have been studied in the literature,
including the weak, Kazhdan-Lusztig, geometric and chain orders
(see \cite{Melnikov1,Melnikov2,Melnikov3,Tas2006}). On the other hand, different partial orders on SYTs and Littlewood-Richardson tableaux (LR-tableaux) have been successfully used to study algebraic and geometric properties of invariant subspaces of nilpotent linear operators (see \cite{ks-JA,ks-MZ,kst2019,ks-tran,kk-generic}). These connections provide a bridge between combinatorics of tableaux and the algebraic structure of invariant subspaces of nilpotent operators, where partial orders encode degeneration, homomorphism and extension phenomena. The results obtained here are applied in \cite{kk-invariant}, where they are combined with homomorphisms and extensions of invariant subspaces of nilpotent linear operators.

In this paper we generalize the results of \cite{kst2019}, where two partial orders on SYTs (and LR-tableaux) of a~fixed shape are studied. In contrast to \cite{kst2019}, we remove the fixed shape assumption and consider the full set of Standard Young Tableaux of a~given size.

Let $\alpha$ be a~partition of weight $\lvert \alpha \rvert = r$, and let $\mathcal{T}_\alpha$ be the set of all SYTs of shape $\alpha$ (see Section \ref{orders_SYT}). In \cite{kst2019}, two partial orders $\boxleq$ and $\domleq$ on $\mathcal{T}_\alpha$ are investigated and their connections with invariant subspaces of nilpotent linear operators are established (see also \cite{ks-MZ}). We extend these orders to the set 
$$ 
\mathcal{T}_r=\bigcup_{\alpha:\, \lvert \alpha \rvert = r}\mathcal{T}_\alpha
$$
and prove the following theorem.

\begin{theorem}\label{thm-main}
For any integer $r\geq 1$, we have $$\boxleq=\domleq$$ on $\mathcal{T}_r$. 
\end{theorem}

To relate our results to the theory of invariant subspaces of nilpotent linear operators, we also consider Littlewood--Richardson tableaux. There is a well-known correspondence between short exact sequences of nilpotent linear operators and LR-tableaux, given by Green-Klein theorem, which assigns to each embedding an LR-tableau describing its combinatorial type (see \cite{macd,klein68}). In \cite{kst2019,ks-MZ,ks-tran,kk-generic}, partial orders on LR-tableaux (in particular for skew shapes that are rook strips) are used to study algebraic and geometric properties of such subspaces, including generic extensions, degeneration, and the ext- and hom-order. Moreover, for rook strips, there exists a natural bijection between LR-tableaux and SYTs that preserves both the dominance order and the box order. As a consequence, our main theorem extends to Littlewood–Richardson tableaux of rook strip shape, providing a~combinatorial framework for studying invariant subspaces via partial orders. \medskip

The partial orders considered on $\mathcal{T}_r$ include all orders presented in \cite{Tas2006}. It would be interesting to understand whether these orders admit a deeper algebraic or geometric interpretation,
and to explore possible connections with structures arising in representation theory and related areas.\medskip


The paper is organized as follows.

In Section~\ref{orders_SYT}, we recall notation and definitions related to Standard Young Tableaux and the orders $\boxleq$ and $\domleq$ on $\mathcal{T}_\alpha$. We then extend these orders to $\mathcal{T}_r$ and prove our main result, namely that $\domleq=\boxleq$.

In Section~\ref{sec-poset}, we investigate structural properties of the poset $(\mathcal{T}_r,\boxleq)$. In particular, we prove that it has unique minimal and maximal elements and that the Hasse diagram admits a center of symmetry. 

Finally, in Section~\ref{seq-SYT-LR}, assuming that $\beta\setminus\gamma$ is a rook strip, we define the orders $\boxleq$ and $\domleq$ on the set $\mathcal{T}_{r,\gamma}^\beta$ of LR-tableaux of shape $\beta\setminus\gamma$ and size $r$. We show that both orders are preserved under the bijection
\[
\Phi^\beta_\gamma :\mathcal{T}_{r,\gamma}^\beta\to\mathcal{T}_r,
\]
and conclude that $\boxleq=\domleq$ on $\mathcal{T}^\beta_{r,\gamma}$.

\section{Two partial orders on SYTs}\label{orders_SYT}

In this section, we present the necessary definitions and notation, and prove the main theorem.

\subsection{Partitions and Standard Young Tableaux}

A \emph{partition} $\al=(\al_1,\ldots,\al_s)$ is a~finite, weakly decreasing sequence of positive integers. The number $s$ is called the \emph{length} of the partition $\al$, and the sum $\lvert \alpha \rvert =\al_1+\ldots+\al_s$ is called its \emph{weight}. Given a~natural number $r$, we denote by $\mathcal{P}_r$ the set of all partitions of weight $r$ and by $\mathcal{P}$ the set of all partitions. When needed, we extend a partition by adding a suitable number of zeros at the end.

A~partition $\al$ is represented by its \emph{Young diagram}, consisting of $s$ top-aligned columns with $\al_1,\ldots,\al_s$ cells, respectively. Since there is an obvious bijection between the set of all Young diagrams and the set of all partitions, the Young diagram connected with a~partition $\al$ is also denoted by $\al$. 
By $\overline{\al}=(\overline{\al}_1,\ldots,\overline{\al}_{\al_1})$ we denote the~partition \emph{transposed} to a~partition $\al$, i.e. the Young diagram of $\overline{\al}$ is the~transposition of the Young diagram $\al$: 
	\[
	\overline{\al}_j=\#\{i\,:\;\al_i\geq j\}.
	\]

The picture below shows an example of two Young diagrams associated with partitions $\al=(3,2,2,1)$ and $\overline{\al}=(4,3,1)$, respectively.

	\[
	\raisebox{.5cm}{$\al\,:$} \quad
		\begin{picture}(12,12)
		\multiput(0,3)(0,3)3{\smbox}
		\multiput(3,6)(0,3)2{\smbox}
		\multiput(6,6)(0,3)2{\smbox}
 		\put(9,9){\smbox}
		\end{picture}\qquad
	\raisebox{.5cm}{$\overline{\al}\,:$} \quad
		\begin{picture}(12,12)
		\multiput(0,0)(0,3)4{\smbox}
		\multiput(3,3)(0,3)3{\smbox}
 		\put(6,9){\smbox}
		\end{picture}
	\] 

\begin{rem} In all diagrams we number rows starting from top and going down.
\end{rem}

We say that partitions $\alpha,\beta\in\mathcal{P}_r$ of the same weight $r$ are in \emph{the natural order} $\alpha \natleq \beta$
if for any $c$ there is
$$\sum\limits_{i=0}^c \alpha_i \geq \sum\limits_{i=0}^c \beta_i$$  or equivalently
$$\sum\limits_{i=0}^c \overline{\alpha}_i \leq \sum\limits_{i=0}^c \overline{\beta}_i,$$ see 
\cite[1.9]{macd}. It is straightforward to check  that $\natleq$ is a~partial order on $\mathcal{P}_r$.                                                                                        
                                                                                    
Assume that $\ga \subseteq \be$. The \emph{skew diagram} $\be\setminus\ga$ is obtained by removing the Young diagram of $\ga$ from that of $\be$, with their upper-left corners aligned. The skew diagram $\be\setminus\ga$ is said to be a~\emph{horizontal strip} if $\be_i\leq\ga_i+1$ holds for each $i=1,\ldots,s$ and a \emph{vertical strip} if $\overline{\be}\setminus\overline{\ga}$ is a~horizontal strip. A~\emph{rook strip} is both vertical and horizontal strip. The picture below shows an example of the horizontal strip $\be\setminus\ga$ and the vertical strip $\overline{\be}\setminus\overline{\ga}$, where $\be=(3,2,2,1)$ and $\ga=(2,1,1)$.

	\[
	\raisebox{.5cm}{$\be\setminus\ga\,:$} \quad
		\begin{picture}(12,12)
		\put(0,3){\smbox}
		\put(3,6){\smbox}
		\put(6,6){\smbox}
 		\put(9,9){\smbox}
		\end{picture}\qquad
	\raisebox{.5cm}{$\overline{\be}\setminus\overline{\ga}\,:$} \quad
		\begin{picture}(12,12)
		\put(0,0){\smbox}
		\multiput(3,3)(0,3)2{\smbox}
 		\put(6,9){\smbox}
		\end{picture}
	\]

\begin{defin}
Let $\al$ be a partition.   A~\emph{Standard Young Tableau} of shape $\al$  (SYT($\al$) for brevity) is the Young diagram $\al$ filled with natural numbers $1,\ldots,\lvert \alpha \rvert $ that strictly increase along rows and down columns. Denote by $\mathcal{T}_\alpha$
the set of all SYT($\al$).
\end{defin}

\begin{rem}
A~SYT($\al$) $\Sigma$ can be identified with the 
 sequence
$\Sigma=[\sigma^{(1)},\ldots,\sigma^{(\lvert \al\rvert )}]$ of partitions such that
$\sigma^{(1)}\subseteq \ldots\subseteq \sigma^{(\lvert \al\rvert )}$, where 
 the skew diagram
$\sigma^{(e)}\setminus\sigma^{(e-1)}$ carries the entry $\singlebox e$ for all $e=1,\ldots,\lvert \al\rvert $. We set $\sigma^{(0)}=(0)$.
\end{rem}

\begin{ex}
Let $\alpha=(3,2,2)$ be a partition and $\Sigma$ be SYT$(\alpha)$ of the form
$$\Sigma = \ytableausetup{centertableaux}\ytableausetup{smalltableaux}
\ytableaushort
{124,367,5}*{3,3,1}$$
The tableau $\Sigma$ can be identified with the 
 sequence
$$\Sigma=[\sigma^{(1)},\ldots,\sigma^{(7)}],$$
where $\sigma^{(1)}=(1),\;\sigma^{(2)}=(1,1),\;\sigma^{(3)}=(2,1),\;\sigma^{(4)}=(2,1,1),\;\sigma^{(5)}=(3,1,1),\;$ $\sigma^{(6)}=(3,2,1),\;\sigma^{(7)}=(3,2,2).$
\end{ex}

Given $r\in \mathbb{N}$ we use the notation
$$ 
\mathcal{T}_r=\bigcup_{\alpha;\lvert \al\rvert =r}\mathcal{T}_\al.
$$

\subsection{Two partial orders on $\mathcal{T}_r$}

\begin{defin}
 \begin{enumerate}
  \item Let $\Sigma=[\sigma^{(1)},\ldots,\sigma^{(r)}],
  \Pi=[\pi^{(1)},\ldots,\pi^{(r)}]\in \mathcal{T}_r$. We say that $\Pi\leq _{\rm dom} \Sigma$ in \emph{the dominance order} if 
  $\pi^{(i)}\leq_{\rm nat}\sigma^{(i)}$ for all $i=1,\ldots,r$. {\bf Note} that we do not require that $\Sigma$ and $\Pi$ have the same shape.
  \item Suppose $\Sigma,\Pi\in \mathcal{T}_r$.
We say that $\Pi$ is obtained from $\Sigma$ by {\it a~decreasing box move} if
$\Pi$ is obtained by swapping two entries in $\Sigma$ so that the smaller entry
  moves to a~higher-numbered row (in this case the SYTs $\Sigma,\Pi$ have the same shape); or if
$\Pi$ is obtained by moving up the last box of a~row of $\Sigma$ to the last position of a~higher-numbered row (in this case the SYTs $\Sigma,\Pi$ have different shapes).  
We denote by $\leq_{\sf box}$ the partial order generated by all box moves.
 \end{enumerate}
\end{defin}

\begin{rem}
 In \cite{kst2019} the partial orders $\leq_{\rm dom}$ and $\leq_{\rm box}$ are defined and investigated on the set
 $\mathcal{T}_\al$ for a~given partition $\al$.
\end{rem}

\begin{ex}
Consider the following example. Let 
\[\Pi= \ytableausetup{centertableaux}\ytableausetup{smalltableaux}
\ytableaushort
{13,25,4}*{2,2,1}\;\;{\rm and}\;\;\Sigma= \ytableausetup{centertableaux}\ytableausetup{smalltableaux}
\ytableaushort
{123,45}*{3,2}\]
The tableaux $\Sigma$ and $\Pi$ can be identified with the 
 sequences
\[\Pi=[\pi^{(1)},\ldots,\pi^{(5)}],\Sigma=[\sigma^{(1)},\ldots,\sigma^{(5)}],\]
where $\pi^{(1)}=(1),\;\pi^{(2)}=(2),\;\pi^{(3)}=(2,1),\;\pi^{(4)}=(3,1),\;\pi^{(5)}=(3,2)$ and
$\sigma^{(1)}=(1),\;\sigma^{(2)}=(1,1),\;\sigma^{(3)}=(1,1,1),\;\sigma^{(4)}=(2,1,1),\;\sigma^{(5)}=(2,2,1).$
It is easy to check that $\pi^{(i)}\natleq \sigma^{(i)}$ for all $i$, so $\Pi\domleq \Sigma$. Moreover we have the following sequence of decreasing box moves transforming tableau $\Sigma$ into tableau $\Pi$:
\[\Pi= \ytableausetup{centertableaux}\ytableausetup{smalltableaux}
\ytableaushort
{13,25,4}*{2,2,1}\boxl \ytableaushort
{12,35,4}*{2,2,1}\boxl \ytableaushort
{124,35}*{3,2}\boxl \ytableaushort
{123,45}*{3,2}=\Sigma,\]
First in $\Sigma$ we swap entries $3$ and $4$, then we move box with entry $4$ from the first to the  third row and finally (to get $\Pi$) we swap entries $2$ and $3$. It follows that $\Pi\boxleq \Sigma$.
\end{ex}

We now prove that $\domleq$ implies $\boxleq$. Let $r\in \mathbb{N}$ and let $\underline{r}$ denote the partition 
$\underline{r}=(r,r,\ldots,r)$ of length
$r$ with all entries equal to $r$. We define a~function
\begin{equation}\label{eq-f-function}
 f:\mathcal{T}_r\to \mathcal{T}_{\underline{r}}
\end{equation}
which allows us to apply the results of \cite{kst2019}.
\noindent
	Before giving the formal definition, we explain the idea. Start with a~tableau $\Pi \in \mathcal{T}_r$. We add the numbers $r+1, r+2, \ldots, r^2$ one by one. At each step, we extend the first row whose length is still smaller than $r$. We continue this process until we obtain the rectangular shape $(r^r)$. 
		Let 	
		\[\Pi= \ytableausetup{centertableaux}\ytableausetup{smalltableaux}
		\ytableaushort
		{13,25,4}*{2,2,1}
		\]
		Then	
		\[f(\Pi)= \ytableausetup{centertableaux}\ytableausetup{smalltableaux}
		\ytableaushort
		{13678,259{10}{11},4{12}{13}{14}{15},{16}{17}{18}{19}{20},{21}{22}{23}{24}{25}}*{5,5,5,5,5}
		\]

The definition of $f$ is stated formally as follows. Let $\Pi=[\pi^{(1)},\ldots,\pi^{(r)}]\in \mathcal{T}_r$. For any $s=0,\ldots, r^2-r-1$  we define recursively the transposed partition of $\pi^{(r+s)}$:
\begin{equation}
 (\overline{\pi^{(r+s+1)}})_i=\left\{ 
 \begin{array}{ll}
  (\overline{\pi^{(r+s)}})_i & \mbox{ for } i\neq j_s+1 \\
  (\overline{\pi^{(r+s)}})_i+1 & \mbox{ for } i= j_s+1
 \end{array}
 \right.
\end{equation}
where $j_s$ is equal to the number of rows of length $r$ in the Young diagram of $\pi^{(r+s)}$. Formally by $j_s=j_s^\Pi\in\{0,1,\ldots,r-1\}$ we denote  the maximal number with the property
\begin{equation}\label{eq-delta-prop}
 (\overline{\pi^{(r+s)}})_{ j_s}=r
\end{equation}
or if there is no number with the property, we set $j_s=0$.

Note that  $\overline{\pi^{(r+s+1)}}$ is a~partition, for all $s$. 
	Indeed, to get $\overline{\pi^{(r+s+1)}}$ we increase by $1$ only the first part of $\overline{\pi^{(r+s)}}$ which is strictly smaller than $r$, hence the weakly decreasing order of parts is preserved.

We set \begin{equation}
f(\Pi)=[\pi^{(1)},\ldots,\pi^{(r)},\pi^{(r+1)},\ldots,\pi^{(r^2)}].        
       \end{equation}
 It is easy to see that $f(\Pi)\in \mathcal{T}_{\underline{r}}$.

\begin{lemma}\label{lem-dom-box}
 Let $\Sigma=[\sigma^{(1)},\ldots,\sigma^{(r)}], \Pi=[\pi^{(1)},\ldots,\pi^{(r)}]\in \mathcal{T}_r$. Then
 \begin{enumerate}
  \item $\Pi\leq_{\rm dom}\Sigma$ if and only if $f(\Pi)\leq_{\rm dom}f(\Sigma)$,
  \item if $f(\Pi)\leq_{\rm box}f(\Sigma)$, then $\Pi\leq_{\rm box}\Sigma$,
  \item if $\Pi\leq_{\rm dom} \Sigma$, then $\Pi\leq_{\rm box}\Sigma$.
 \end{enumerate}
\end{lemma}

\begin{proof}
 1. Assume that $\Pi\leq_{\rm dom}\Sigma$. It follows that $\pi^{(i)}\leq_{\rm nat}\sigma^{(i)}$ for all $i=1,\ldots,r$. By the induction on $s$ we prove that $\pi^{(r+s+1)}\leq_{\rm nat}\sigma^{(r+s+1)}$ for all $s=-1,0,\ldots,r^2-r-1$. For $s=-1$, the statement is immediate. Let $s\geq 0$ and let $l_s=j_s^\Sigma,j_s=j_s^\Pi$ be as above, see (\ref{eq-delta-prop}). 
 
 Note that $j_s\leq l_s$. Indeed, if $j_s=0$, we are done. Otherwise, by the choice of $j_s$ we have
 $(\overline{\pi^{(r+s)}})_{j_s+1}<r$ and
 $(\overline{\pi^{(r+s+1)}})_{i}=(\overline{\pi^{(r+s)}})_{i}=r$ for $1\leq i\leq j_s$. Therefore by the induction we get:
 \begin{equation}
  \label{eq-js-ls}
  j_s\cdot r=\sum_{i=1}^{j_s}(\overline{\pi^{(r+s)}})_{i}\leq \sum_{i=1}^{j_s}(\overline{\sigma^{(r+s)}})_{i}\leq j_s\cdot r.
 \end{equation}
The last inequality is true, because all parts of all partitions we consider are less or equal to $r$. It follows from (\ref{eq-js-ls}) that
$(\overline{\sigma^{(r+s)}})_{i}=r$ for all $i=1,\ldots,j_s$. By the choice of $l_s$, we get $l_s\geq j_s$.

Note that 
\begin{equation}
 \sum_{i=1}^t(\overline{\sigma^{(r+s+1)}})_i=\left\{ 
 \begin{array}{ll} \label{eq-gamma}
  \sum_{i=1}^t(\overline{\sigma^{(r+s)}})_i & \mbox{ for } t\leq l_s \\
  \sum_{i=1}^t(\overline{\sigma^{(r+s)}})_i+1 & \mbox{ for }  t\geq l_s+1
 \end{array}
 \right.
\end{equation}
and
\begin{equation}\label{eq-delta}
 \sum_{i=1}^t(\overline{\pi^{(r+s+1)}})_i=\left\{ 
 \begin{array}{ll}
  \sum_{i=1}^t(\overline{\pi^{(r+s)}})_i & \mbox{ for }  t\leq j_s \\
  \sum_{i=1}^t(\overline{\pi^{(r+s)}})_i+1 & \mbox{ for } t\geq j_s+1
 \end{array}
 \right.
\end{equation}
For $t\leq j_s$ and $t>l_s$, by the induction and (\ref{eq-gamma}), (\ref{eq-delta}) we get
$$
\sum_{i=1}^t(\overline{\pi^{(r+s+1)}})_i\leq \sum_{i=1}^t(\overline{\sigma^{(r+s+1)}})_i.
$$
If $j_s=l_s$ we are done.

Let $j_s<j_s+1\leq t\leq l_s$.  
By the choice of $j_s,l_s$ and the fact that 
$j_s<l_s$, we have $(\overline{\pi^{(r+s)}})_{i} <(\overline{\sigma^{(r+s)}})_{i}=r$ for all $i=j_s+1,\ldots,l_s$. Moreover by the induction hypothesis we have $
\sum_{i=1}^{j}(\overline{\pi^{(r+s)}})_i\leq \sum_{i=1}^{j}(\overline{\sigma^{(r+s)}})_i,
$
for all $j$.
Therefore we get 
$
\sum_{i=1}^{t}(\overline{\pi^{(r+s)}})_i< \sum_{i=1}^{t}(\overline{\sigma^{(r+s)}})_i
$ for $t$ such that $j_s<j_s+1\leq t\leq l_s$.
Applying the formulae (\ref{eq-gamma}), (\ref{eq-delta}) we deduce that $
\sum_{i=1}^t(\overline{\pi^{(r+s+1)}})_i\leq \sum_{i=1}^t(\overline{\sigma^{(r+s+1)}})_i
$ for all $t$  such that $j_s<j_s+1\leq t\leq l_s$. Finally, $f(\Pi)\leq_{\rm dom}f(\Sigma)$. Since the converse implication is obvious we are done.\medskip

To prove the statement 2, consider a sequence of decreasing box moves in $\mathcal T_{\underline r}$:
\[
f(\Sigma)=T_0 \succ_{\rm box} T_1 \succ_{\rm box}\cdots \succ_{\rm box} T_m=f(\Pi).
\]	 
For each $k$, let $\widetilde T_k$ be the tableau obtained from $T_k$ by deleting all boxes labelled by numbers greater than $r$. Then $\widetilde T_0=\Sigma$, $\widetilde T_m=\Pi$, and $\widetilde T_k\in\mathcal T_r$. It suffices to check that each move $T_k\succ_{\rm box}T_{k+1}$ induces either the same tableau or a decreasing box move on $\widetilde T_k$. If both entries involved in the move are greater than $r$, then $\widetilde T_k=\widetilde T_{k+1}$. 
If both are at most $r$, the same swap appears in $\widetilde T_k$, hence it is a decreasing box move. 
Finally, if one entry is at most $r$ and the other is greater than $r$, then after deleting the latter, the move corresponds to shifting a~box to a~higher numbered row. Indeed, the box with the entry $\le r$ moves to the position previously occupied by the larger entry, which lies in a higher numbered row. Hence the induced move on $\widetilde T_k$ is a decreasing box move of the second type. Thus, after removing repetitions, we obtain a sequence of decreasing box moves from $\Sigma$ to $\Pi$, and therefore $\Pi\le_{\rm box}\Sigma$.
 \medskip

The statement 3 is a~consequence of 1, 2, because by \cite[Theorem 1.1]{kst2019} we have:
$f(\Pi)\leq_{\rm dom} f(\Sigma)$ if and only if
$f(\Pi)\leq_{\rm box} f(\Sigma)$.
\end{proof}

{\it Proof of Theorem \ref{thm-main}.} Fix an integer $r\geq 1$. By Lemma \ref{lem-dom-box} we have $\domleq\subseteq \boxleq$.

Let $\Pi,\Sigma\in\mathcal{T}_r$. Suppose that $\Pi$ is obtained from $\Sigma$ by a single decreasing box move. 
If $\Pi$ arises from $\Sigma$ by swapping two entries, then $\Pi$ and $\Sigma$ have the same shape, and $\Pi\domleq \Sigma$ by \cite[Theorem 1.1]{kst2019}.

Now consider the move that shifts the box $\singlebox s$ from the last position of row $n$ of $\Sigma$ to the last position of row $m>n$. 
Then $\sigma^{(1)}=\pi^{(1)},\ldots,\sigma^{(s-1)}=\pi^{(s-1)}$, and for $j\ge s$ we have
\[
(\overline{\pi^{(j)}})_i=
\begin{cases}
\overline{\sigma^{(j)}}_i & \text{if } i\notin\{n,m\},\\
\overline{\sigma^{(j)}}_i-1 & \text{if } i=n,\\
\overline{\sigma^{(j)}}_i+1 & \text{if } i=m.
\end{cases}
\]

Since $n<m$, it follows that $\pi^{(j)}\natleq\sigma^{(j)}$, and hence $\Pi\domleq \Sigma$.
Therefore, a single decreasing box move implies $\Pi\domleq \Sigma$. Since the order $\boxleq$ is generated by box moves, $\Pi\boxleq\Sigma$ implies $\Pi\domleq\Sigma$. This finishes the proof of Theorem \ref{thm-main}.
\medskip

	Combining the above results, we obtain the following corollary.
	\begin{cor}
		Let $\Sigma=[\sigma^{(1)},\ldots,\sigma^{(r)}], \Pi=[\pi^{(1)},\ldots,\pi^{(r)}]\in \mathcal{T}_r$. Then
		 $\Pi\leq_{\rm box}\Sigma$ if and only if $f(\Pi)\leq_{\rm box}f(\Sigma)$.
	\end{cor}


\section{Properties of the order $\boxleq$}\label{sec-poset}
In this section, we examine some combinatorial properties of the poset $(\mathcal{T}_r,\boxleq)$. 

	\begin{ex}
		The map $f$ does not preserve cover relations. Let $r=3$ and consider
		\[
		\Sigma=\ytableausetup{smalltableaux}\ytableaushort{13,2},
		\qquad
		\Pi=\ytableaushort{1,2,3}.
		\]
		Then $\Pi$ is obtained from $\Sigma$ by moving the last box of the first row to a lower row, hence
		\[
		\Pi \prec_{\rm box} \Sigma
		\]
		in $\mathcal T_3$. On the other hand,
		\[
		f(\Sigma)=\ytableaushort{134,256,789},
		\qquad
		f(\Pi)=\ytableaushort{145,267,389}.
		\]
		In $\mathcal T_{(3,3,3)}$ we have the chain
		\[
		\ytableaushort{145,267,389}
		\prec_{\rm box}
		\ytableaushort{135,267,489}
		\prec_{\rm box}
		\ytableaushort{134,267,589}
		\prec_{\rm box}
		\ytableaushort{134,257,689}
		\prec_{\rm box}
		\ytableaushort{134,256,789}.
		\]
		Thus $f(\Pi)$ and $f(\Sigma)$ are not comparable by a single box move, so $f$ does not preserve covers.
	\end{ex}

\begin{rem}
 The Hasse diagram of the poset $(\mathcal{T}_5,\boxleq)$ is presented on page \pageref{hasseT5}.
\begin{figure}[p]
	\centering
	$$\scalebox{0.8}[0.8]{\setlength{\unitlength}{0.075cm}
		\begin{picture}(110,300)(50,-150) 
			\put(45,-135){\ytableausetup{smalltableaux}\ytableaushort{1,2,3,4,5}* {}    
			}      
			\put(45,-105){\ytableaushort{15,2,3,4}* {}    
			}     
			\put(45,-75){\ytableaushort{14,2,3,5}* {}    
			}   
			\put(30,-55){\ytableaushort{13,2,4,5}* {}    
			}  
			\put(60,-55){\ytableaushort{14,25,3}* {}    
			}   
			\put(15,-35){\ytableaushort{12,3,4,5}* {}    
			}      
			\put(45,-35){\ytableaushort{13,25,4}* {}    
			}  
			\put(75,-35){\ytableaushort{145,2,3}* {}    
			} 
			\put(30,-15){\ytableaushort{12,35,4}* {}    
			}  
			\put(60,-15){\ytableaushort{13,24,5}* {}    
			} 
			\put(90,-15){\ytableaushort{135,2,4}* {}    
			} 
			\put(45,15){\ytableaushort{12,34,5}* {}    
			} 
			\put(75,15){\ytableaushort{134,2,5}* {}    
			} 
			\put(105,15){\ytableaushort{125,3,4}* {}    
			}
			\put(135,15){\ytableaushort{135,24}* {}    
			}
			\put(92,15){ $\bullet$ }
			\put(90,45){\ytableaushort{124,3,5}* {}    
			} 
			\put(120,45){\ytableaushort{125,34}* {} 
			}   
			\put(150,45){\ytableaushort{134,25}* {}    
			} 
			\put(105,65){\ytableaushort{123,4,5}* {}    
			} 
			\put(135,65){\ytableaushort{124,35}* {} 
			}   
			\put(165,65){\ytableaushort{1345,2}* {}    
			}
			
			\put(120,85){\ytableaushort{123,45}* {} 
			}   
			\put(150,85){\ytableaushort{1245,3}* {}    
			}

			\put(135,105){\ytableaushort{1235,4}* {}    
			}
			
			\put(135,125){\ytableaushort{1234,5}* {}    
			}
			
			\put(135,145){\ytableaushort{12345}* {}    
			}
			\put(48,-113){\line(0,-1){6}}  
			
			\put(48,-85){\line(0,-1){9}}  
			
			\put(33,-65){\line(1,-1){9}}
			\put(63,-65){\line(-1,-1){9}}

			\put(18,-45){\line(1,-1){9}}
			\put(47,-45){\line(-1,-1){9}} 
			
			\put(49,-45){\line(1,-1){9}}
			\put(78,-45){\line(-1,-1){9}}

			\put(32,-23){\line(-1,-1){9}}   
			\put(34,-23){\line(1,-1){9}}
			\put(63,-23){\line(-1,-1){9}}  
			\put(97,-23){\line(-1,-1){9}}
			\put(88,-18){\line(-5,-2){35}}
			
			\put(47,10){\line(-1,-2){8}}   
			\put(51,10){\line(1,-2){8}}
			\put(78,10){\line(-1,-2){8}}   
			\put(83,10){\line(1,-2){8}}
			
			\put(112,10){\line(-1,-2){8}}
			\put(41,-11){\line(5,2){63}}
			\put(71,-11){\line(5,2){63}}
			\put(104,-11){\line(5,3){30}}

			\put(92,40){\line(-1,-2){8}}   
			\put(98,40){\line(1,-2){8}}
			\put(123,40){\line(-1,-2){8}}   
			\put(128,40){\line(1,-2){8}}
			\put(152,40){\line(-1,-2){8}}
			
			\put(58,19){\line(5,2){60}}
			\put(88,19){\line(5,2){60}}
			\put(59,24){\line(5,3){30}}

			\put(107,60){\line(-1,-1){9}}   
			\put(138,60){\line(-1,-1){9}}  
			\put(139,60){\line(1,-1){9}}  
			\put(172,60){\line(-1,-1){9}}
			\put(137,60){\line(-5,-2){35}}
			
			\put(122,81){\line(-1,-1){8}}   
			\put(124,81){\line(1,-1){9}}
			\put(153,81){\line(-1,-1){9}}  
			\put(154,81){\line(1,-1){9}}

			\put(137,101){\line(-1,-1){9}}   
			\put(139,101){\line(1,-1){9}}
			
			\put(138,120){\line(0,-1){8}}   
			\put(138,142){\line(0,-1){9}}
	\end{picture}}$$
	\caption{The Hasse diagram of the poset $(\mathcal{T}_5,\boxleq)$.}
\end{figure}\label{hasseT5}   
We can observe the following properties of $(\mathcal{T}_5,\boxleq)$:
\begin{itemize}
	\item there exists exactly one maximal element and exactly one minimal element,
	\item all saturated chains have the same length,
	\item this poset is not a lattice,
	\item the poset is symmetric with respect to the point indicated in the middle of the diagram.
\end{itemize}     
However, the poset $(\mathcal{T}_5,\boxleq)$ is the largest poset such that all saturated chains have the same length. For $r\geq 6$, this property fails, as illustrated below. Here are two saturated chains in $(\mathcal{T}_6,\boxleq)$ with the same endpoints and of lengths $3$ and $4$, respectively:
\begin{center}
$\ytableausetup{centertableaux}\ytableausetup{smalltableaux}\ytableaushort{12,35,4,6}*{2,2,1,1}
\boxleq
\ytableaushort{125,3,4,6}*{3,1,1,1}
\boxleq
\ytableaushort{125,36,4}*{3,2,1}$

$\ytableausetup{centertableaux}\ytableausetup{smalltableaux}\ytableaushort{12,35,4,6}*{2,2,1,1}
\boxleq
\ytableaushort{12,35,46}*{2,2,2}
\boxleq
\ytableaushort{126,35,4}*{3,2,1}
\boxleq
\ytableaushort{125,36,4}*{3,2,1}$
\end{center}
Similarly, for any $n>6$, one can obtain saturated chains in $(\mathcal{T}_n,\boxleq)$ with common endpoints and lengths $3$ and $4$. These are constructed by replacing $6$ by $n$ and then adding, in each tableau, the cells with entries $6,\ldots,n-1$ in the first column.

\begin{center}
	$
	\begin{ytableau}
		1 & 2  \\
		3 & 5 \\
		4 \\
		\none[\vdots] \\
		\hat{n} \\
		n
	\end{ytableau}
	\boxleq
	\begin{ytableau}
		1 & 2  & 5 \\
		3 \\
		4 \\
		\none[\vdots] \\
		\hat{n} \\
		n
	\end{ytableau}
	\boxleq
	\begin{ytableau}
		1 & 2  & 5 \\
		3 &n\\
		4 \\
		\none[\vdots] \\
		\hat{n} 
	\end{ytableau}
	$
	
	$
	\begin{ytableau}
		1 & 2  \\
		3 & 5 \\
		4 \\
		\none[\vdots] \\
		\hat{n} \\
		n
	\end{ytableau}
	\boxleq
	\begin{ytableau}
		1 & 2   \\
		3 & 5\\
		4 & n\\
		\none[\vdots] \\
		\hat{n} 
	\end{ytableau}
	\boxleq
	\begin{ytableau}
		1 & 2 &n   \\
		3 & 5\\
		4 \\
		\none[\vdots] \\
		\hat{n} 
	\end{ytableau}
	\boxleq
	\begin{ytableau}
		1 & 2  & 5 \\
		3 &n\\
		4 \\
		\none[\vdots] \\
		\hat{n} 
	\end{ytableau}
	$
	\vspace{1em}
\end{center}
Here, $\hat{n}$ denotes $n-1$.
\end{rem}

\begin{lemma}\label{lem-hasse-symetry}
		For a natural number $r$, the transposition of tableaux defines an order-reversing involution on the poset $(\mathcal{T}_r,\boxleq)$. Formally, if $\Pi\boxleq \Sigma$, then $\overline{\Sigma}\boxleq \overline{\Pi}$ for $\Pi,\Sigma\in\mathcal{T}_r$,  where $\overline{\Pi}=[\overline{\pi^{(1)}},\ldots,\overline{\pi^{(r)}}]$ is obtained by transposing each partition in the sequence defining $\Pi$.
\end{lemma}
\begin{proof}
Let $\Pi,\Sigma\in\mathcal{T}_r$ be such that $\Pi\boxleq \Sigma$. We fix sequences of partitions $\Pi = [\pi^{(1)},\ldots,\pi^{(r)}]$ and $\Sigma=[\sigma^{(1)},\ldots,\sigma^{(r)}]$. Then $\overline{\Pi} = [\overline{\pi^{(1)}},\ldots,\overline{\pi^{(r)}}]$ and $\overline{\Sigma}=[\overline{\sigma^{(1)}},\ldots,\overline{\sigma^{(r)}}]$. By Theorem \ref{thm-main} we obtain $\pi^{(i)}\natleq\sigma^{(i)}$ for all $1\le i\le r$. Consequently, $\overline{\sigma^{(i)}}\natleq\overline{\pi^{(i)}}$ for all $1\le i\le r$. Hence, by definition of $\domleq$, we have $\overline{\Sigma}\domleq \overline{\Pi}$. Lemma \ref{lem-dom-box} 3. now yields $\overline{\Sigma}\boxleq \overline{\Pi}$.
\end{proof}

\begin{lemma}
Let $r$ be a natural number.
\begin{enumerate}		
\item There is a unique maximal element in the poset $(\mathcal{T}_r,\boxleq)$, namely the one row SYT with $r$ cells.
		
\item There is a unique minimal element in the poset $(\mathcal{T}_r, \boxleq )$, namely the  one column SYT with $r$ cells.
\end{enumerate}	
\end{lemma}

\begin{proof}  1. Let $\Delta\in\mathcal T_r$ be the one row SYT with $r$ cells. Then $\Delta=[\delta^{(1)},\ldots,\delta^{(r)}]$, where   $\delta^{(i)}=(1,\ldots,1)$ is the partition of weight $i$ consisting of $i$ parts equal to $1$, for each $i=1,\ldots,r$.  We show that $\Sigma\domleq \Delta$ for every $\Sigma=[\sigma^{(1)},\ldots,\sigma^{(r)}]\in\mathcal T_r$. Since $\sigma^{(i)}$ is a partition of weight $i$, we have
\[
	\sum_{j=1}^c \sigma^{(i)}_j \geq \sum_{j=1}^c \delta^{(i)}_j
	\qquad\text{for all }c,
\]
hence $\sigma^{(i)}\natleq \delta^{(i)}$ for all $i$. Therefore $\Sigma\domleq \Delta$. By Theorem \ref{thm-main}, we obtain $\Sigma\boxleq \Delta$. Thus $\Delta$ is the unique maximal element of $(\mathcal T_r,\boxleq)$. 
	
	2. This follows directly from the statement 1 and Lemma \ref{lem-hasse-symetry}.
\end{proof}

\section{SYTs and LR-tableaux} \label{seq-SYT-LR}

In this section, we extend the results of Section~\ref{orders_SYT} to LR-tableaux in the case of rook strips. This provides a combinatorial framework for studying algebraic properties of invariant subspaces of nilpotent linear operators via box moves and the dominance order, as developed in \cite{ks-JA,ks-MZ}. In particular, in \cite{kk-invariant} we show that the box order corresponds to extensions of invariant subspaces, while the dominance order corresponds to homomorphisms between them.

\begin{defin}
	Let $(\alpha,\beta,\gamma)$ be a~triple of partitions. An {\it LR-tableau} of shape $\beta\setminus\gamma$ and content $\alpha$ is a~Young diagram of shape $\beta$ (the {\it outer shape}) in which the skew diagram $\beta\setminus\gamma$ is filled with $\overline{\alpha}_1$ entries equal to $1$, $\overline{\alpha}_2$ entries equal to $2$, and so on, such that the following conditions are satisfied:
	\begin{enumerate}
		\item entries weakly increase along each row;
		\item entries strictly increase down each column;
		\item the lattice permutation property holds, that is, on the right-hand side of each column there occur at least as many entries $e$ as entries $e+1$ ($e=1,2,\ldots$).
	\end{enumerate}
\end{defin}

Let $\alpha,\beta,\gamma$ be partitions and let $r\in\mathbb{N}$. {\bf Throughout, we assume that $\beta\setminus \gamma$ is a~rook strip.} 

Denote by $\mathcal{T}_{\alpha,\gamma}^\beta$ the set of all LR-tableaux of type $(\alpha,\beta,\gamma)$, and by $\mathcal{T}_{r,\gamma}^\beta$ the set of all such tableaux with $\lvert \alpha \rvert = r$. Then
$$
\mathcal{T}_{r,\gamma}^\beta=\bigcup^\bullet_{\alpha:\,\lvert \alpha \rvert =r}\mathcal{T}_{\alpha,\gamma}^\beta.
$$

\begin{rem}
	An LR-tableau $\Gamma$ of shape $\beta\setminus\gamma$ and content $\alpha$ can be identified with the sequence
	$\Gamma=[\gamma^{(0)},\ldots,\gamma^{(\lvert \alpha \rvert)}]$ of partitions such that
	$\gamma^{(0)}\subseteq \gamma^{(1)}\subseteq \cdots \subseteq \gamma^{(\lvert \alpha \rvert)}$, where the skew diagram
	$\gamma^{(e)}\setminus\gamma^{(e-1)}$ carries all entries $\singlebox e$ for $e=1,\ldots,\lvert \alpha \rvert$. We set $\gamma^{(0)}=\gamma$ and $\gamma^{(\lvert \alpha \rvert)}=\beta$. Note that in our setting the content $\alpha$ may vary. Therefore, the sequence of partitions may stabilize, that is, we may have $\gamma^{(e)}=\gamma^{(e+1)}=\cdots=\beta$ for some $e<\lvert\alpha\rvert$.
\end{rem}

Box order and dominance order were defined on LR-tableaux in \cite{ks-MZ} for tableaux of fixed shape. Here we recall these definitions and extend them to the set $\mathcal{T}_{r,\gamma}^\beta$, allowing tableaux with varying content, similarly as in Section~\ref{orders_SYT}.

  \begin{defin} Let $\beta,\gamma$ be partitions, $r=\lvert \beta \rvert -\lvert \gamma \rvert $ and let $\Delta,\Gamma\in \mathcal{T}_{r,\gamma}^\beta$.
  \begin{enumerate}
   \item
   The tableau $\Delta$ is obtained by a {\it decreasing
    box move} from the tableau
    $\Gamma$ (we write $\Delta\prec_{\rm box}\Gamma$) if $\Delta$ is obtained from $\Gamma$ by swapping two entries, such that the
    smaller entry moves to a higher row, i.e. to a row with a smaller index (this move does not change the partition $\alpha$); or by increasing an~entry (thereby changing the partition $\alpha$). The box-order $\leq_{\rm box}$ is the partial order
    generated by the relation $\prec_{\rm box}$.
 \item Let $\Delta=[\delta^{(0)},\delta^{(1)},\ldots,\delta^{(r)}]\in \mathcal{T}_{\alpha',\gamma}^\beta,\;\Gamma=[\gamma^{(0)},\gamma^{(1)},\ldots,\gamma^{(r)}]\in \mathcal{T}_{\al,\gamma}^\beta$. We say that $\Delta\leq _{\rm dom} \Gamma$ in the dominance order if 
  $\delta^{(i)}\natleq\gamma^{(i)}$ for all $i=0,1,\ldots,r$. 
  \end{enumerate}
  \end{defin}

  \begin{rem}
   Note that  for arbitrary $\Gamma=[\gamma^{(0)},\ldots,\gamma^{(r)}]\in \mathcal{T}_{\al,\gamma}^\beta$  we have 
   $$\sum\limits_{i=0}^c \overline{\alpha}_i=\lvert \gamma^{(c)}\rvert -\lvert \gamma \rvert .$$
   Therefore, if $\Delta\domleq \Gamma$, then  $\alpha\natleq\alpha'$.
  \end{rem}

\begin{ex}
Let 
\[\Delta= \ytableausetup{centertableaux}\ytableausetup{smalltableaux}\ytableaushort
{\none\none\none\none1,\none\none\none2,\none\none1,\none3,2}*{5,4,3,2,1}
\;\;{\rm and}\;\;\Gamma=\ytableaushort
{\none\none\none\none1,\none\none\none1,\none\none1,\none2,2}*{5,4,3,2,1}\]
be LR-tableaux of shape $\beta\setminus\gamma$ and content $\alpha$ and $\alpha'$ respectively for $\gamma=(4,3,2,1),\;$ $\beta=(5,4,3,2,1),\; \alpha=(3,2)$ and $\alpha'=(2,2,1).$
The LR-tableaux $\Delta,\Gamma$ are identified with the 
 sequences
\[\Delta=[\delta^{(0)},\ldots,\delta^{(5)}],\Gamma=[\gamma^{(0)},\ldots,\gamma^{(5)}],\]
where $\delta^{(0)}=\gamma,\;\delta^{(1)}=(4,3,3,1,1),\;\delta^{(2)}=(5,3,3,2,1),\;\delta^{(3)}=\delta^{(4)}=\delta^{(5)}=\beta$ and
$\gamma^{(0)}=\gamma,\;\gamma^{(1)}=(4,3,3,2,1),\;\gamma^{(2)}=\gamma^{(3)}=\gamma^{(4)}=\gamma^{(5)}=\beta.$
One can check that $\delta^{(i)}\natleq \gamma^{(i)}$ for all $i$, so $\Delta\domleq \Gamma$. Moreover, we have following sequence of decreasing box moves transforming tableau $\Gamma$ into tableau 
$\Delta$:
\[\Delta= \ytableausetup{centertableaux}\ytableausetup{smalltableaux}\ytableaushort
{\none\none\none\none1,\none\none\none2,\none\none1,\none3,2}*{5,4,3,2,1}\boxl  \ytableaushort
{\none\none\none\none1,\none\none\none2,\none\none1,\none1,2}*{5,4,3,2,1}\boxl \ytableaushort
{\none\none\none\none1,\none\none\none1,\none\none1,\none2,2}*{5,4,3,2,1}=\Gamma,\]
so $\Delta\boxleq \Gamma$.
\end{ex}

 In \cite[Section~3]{kst2019}, a bijection
 \[
 \Phi^\beta_\gamma : \mathcal{T}_{\alpha,\gamma}^\beta \to \mathcal{T}_\alpha
 \]
 is defined as follows.
 For an LR-tableau $\Gamma$, let $\tau(\Gamma) = (\tau_1,\ldots,\tau_{\lvert\alpha\rvert})$ be its reading word, i.e.  the sequence obtained by reading the entries column by column, from top to bottom, starting with the rightmost column and proceeding to the left.
 
 Fix $i \leq s$, where $s$ is the number of rows of $\alpha$, and let $j_1 < j_2 < \cdots < j_{n_i}$ be all indices $j$ such that $\tau_j = i$. We place the entries $j_1, j_2, \ldots, j_{n_i}$ in the $i$-th row of the corresponding standard Young tableau of shape $\alpha$. This bijection has the property that given $\Delta,\Gamma\in \mathcal{T}_{\al,\gamma}^\beta$ the relation $\Delta\domleq\Gamma$ (resp. $\Delta\boxleq\Gamma$) holds if and only if
$\Phi^\beta_\gamma(\Delta)\domleq\Phi^\beta_\gamma(\Gamma)$ (resp. $\Phi^\beta_\gamma(\Delta)\boxleq\Phi^\beta_\gamma(\Gamma)$), see \cite[Lemma 3.1,Lemma 3.2]{kst2019}. 

The following example is taken from \cite{kst2019}. Consider LR-tableaux related by a decreasing box move swapping entries in rows $3$ and $6$:
\[\Delta= \ytableausetup{centertableaux}\ytableausetup{smalltableaux}\ytableaushort
{\none\none\none\none\none1,\none\none\none\none1,\none\none\none2,\none\none2,\none3,1}*{6,5,4,3,2,1}
\;\;\boxleq\;\;\Gamma=\ytableaushort
{\none\none\none\none\none1,\none\none\none\none1,\none\none\none1,\none\none2,\none3,2}*{6,5,4,3,2,1}
\]

 The the corresponding reading words are $\tau(\Delta)=(1,1,2,2,3,1)$ and $\tau(\Gamma)=(1,1,1,2,3,2)$  and the corresponding SYTs are
\[\Phi^\beta_\gamma(\Delta):= \ytableausetup{centertableaux}\ytableausetup{smalltableaux}\ytableaushort
{126,34,5}*{3,2,1}
\;\; \;\;\boxleq\;\;\;\;\Phi^\beta_\gamma(\Gamma):= \ytableausetup{centertableaux}\ytableausetup{smalltableaux}\ytableaushort
{123,46,5}*{3,2,1}
\]

  Starting from
  $\Phi^\beta_\gamma(\Gamma)$ we can get $\Phi^\beta_\gamma(\Delta)$ by
  first swapping 3 and 4, then 4 and 6.

In the natural way we extend $\Phi_\gamma^\beta$ to the bijection
\begin{equation}\label{eq-phi-gamma-beta}
 \Phi^\beta_\gamma :\mathcal{T}_{r,\gamma}^\beta\to\mathcal{T}_r
\end{equation}

\begin{lemma}\label{lem-LR-SYT-box-dom}
Given $\Delta,\Gamma\in \mathcal{T}_{r,\gamma}^\beta$:
\begin{enumerate}
 \item $\Delta\domleq \Gamma$ if and only if $\Phi^\beta_\gamma(\Delta)\domleq \Phi^\beta_\gamma(\Gamma)$,
 \item $\Delta\boxleq \Gamma$ if and only if $\Phi^\beta_\gamma(\Delta)\boxleq \Phi^\beta_\gamma(\Gamma)$.
\end{enumerate}
\end{lemma}

\begin{proof}
\begin{enumerate}
\item 
For an LR-tableau $\Lambda\in \mathcal{T}_{r,\gamma}^\beta$, denote by $\tau(\Lambda)=(\tau_1,\ldots,\tau_r)$ its reading word.  
For $c,l\in\{1,\ldots,r\}$ set
\[
A^\Lambda_{c,l}:=\#\{j\le l:\tau_j(\Lambda)\le c\}.
\]

By the definition of $\Phi^\beta_\gamma$, the entries lying in the first $c$ rows of $\Phi^\beta_\gamma(\Lambda)$ are precisely the indices $j$ such that $\tau_j(\Lambda)\le c$. Hence
\begin{equation}\label{eq-A}
A^\Lambda_{c,l}
=
\#\{\text{entries }\le l\text{ in the first }c\text{ rows of }\Phi^\beta_\gamma(\Lambda)\}
=
\sum_{i=1}^c \overline{\rho^{(l)}}_i
\end{equation}
where $\Phi^\beta_\gamma(\Lambda)=[\rho^{(1)},\ldots,\rho^{(r)}]$. On the other hand, since $\beta\setminus\gamma$ is a rook strip, each row contains at most one box of the skew diagram. Let $l'$ be the row containing the $l$-th box in the reading order. Then among the first $l$ positions of $\tau(\Lambda)$ we see exactly the boxes of $\Lambda$ lying in the first $l'$ rows of the skew diagram. Therefore the number of those boxes whose entries are at most $c$ is equal to
\[
A^\Lambda_{c,l}
=
\sum_{i=1}^{l'} \overline{\lambda^{(c)}}_i-\sum_{i=1}^{l'}\overline{\lambda^{(0)}}_i,
\]
where
\(
\Lambda=[\lambda^{(0)},\lambda^{(1)},\ldots,\lambda^{(r)}].
\)

Applying this to $\Gamma$ and $\Delta$, where 
\(
\Gamma=[\gamma^{(0)},\gamma^{(1)},\ldots,\gamma^{(r)}]\), \( \Delta=[\delta^{(0)},\delta^{(1)},\ldots,\delta^{(r)}]\), \(
\Phi^\beta_\gamma(\Gamma)=[\pi^{(1)},\ldots,\pi^{(r)}]\), 
\(\Phi^\beta_\gamma(\Delta)=[\sigma^{(1)},\ldots,\sigma^{(r)}],
\)
we obtain
\[
A^\Gamma_{c,l}
=
\sum_{i=1}^{l'} \overline{\gamma^{(c)}}_i-\sum_{i=1}^{l'}\overline{\gamma}_i,
\qquad
A^\Delta_{c,l}
=
\sum_{i=1}^{l'} \overline{\delta^{(c)}}_i-\sum_{i=1}^{l'}\overline{\gamma}_i,
\]
because the inner shape $\gamma=\gamma^{(0)}=\delta^{(0)}$ is the same for both tableaux. Therefore the condition
\[
\delta^{(c)}\natleq \gamma^{(c)}
\quad\text{for all }c
\]
is equivalent to
\[
A^\Delta_{c,l}\le A^\Gamma_{c,l}
\quad\text{for all }c,l.
\]
Since (\ref{eq-A}), this is equivalent to
\[
\sum_{i=1}^c \overline{\sigma^{(l)}}_i
\le
\sum_{i=1}^c \overline{\pi^{(l)}}_i
\quad\text{for all }c,l,
\]
that is,
\[
\sigma^{(l)}\natleq \pi^{(l)}
\quad\text{for all }l.
\]
Hence
\[
\Delta\domleq \Gamma
\quad\Longleftrightarrow\quad
\Phi^\beta_\gamma(\Delta)\domleq \Phi^\beta_\gamma(\Gamma).
\]

\item Let $\Delta,\Gamma\in \mathcal{T}_{r,\gamma}^\beta$. Suppose $\Delta$ is obtained from $\Gamma$ by a single decreasing box move. If $\Delta$ is obtained from $\Gamma$ by swapping the positions of two entries, then $\Delta, \Gamma\in \mathcal{T}_{\alpha,\gamma}^\beta$ for the same partition $\alpha$. Therefore $\Phi^\beta_\gamma(\Delta)\boxleq \Phi^\beta_\gamma(\Gamma)$, by \cite[Lemma 3.1]{kst2019}.

Otherwise, $\Delta$ is obtained from $\Gamma$ by increasing an entry. 
Let $\tau(\Gamma)=(\tau_1(\Gamma),\ldots,\tau_r(\Gamma))$ and $\tau(\Delta)=(\tau_1(\Delta),\ldots,\tau_r(\Delta))$ be the reading words of $\Gamma$ and $\Delta$, respectively. 
Since the move consists of increasing a single entry,  the reading words $\tau(\Gamma)$ and $\tau(\Delta)$ differ in exactly one position. Then there exists an index $s$ such that
\[
\tau_s(\Gamma)=e \quad\text{and}\quad \tau_s(\Delta)=f
\]
for some $e<f$, and $\tau_j(\Gamma)=\tau_j(\Delta)$ for all $j\neq s$. By the definition of $\Phi^\beta_\gamma$, an index $j$ lies in row $i$ of $\Phi^\beta_\gamma(\Gamma)$ if and only if $\tau_j(\Gamma)=i$. 
Thus the entry $\singlebox s$ lies in row $e$ of $\Phi^\beta_\gamma(\Gamma)$ and in row $f$ of $\Phi^\beta_\gamma(\Delta)$, while all other entries lie in the same rows in both tableaux. Hence $\Phi^\beta_\gamma(\Delta)$ can be obtained from $\Phi^\beta_\gamma(\Gamma)$ by moving the entry $\singlebox s$ from row $e$ to row $f$ and reordering the entries in rows $f$ and $e$ if necessary. This reordering step is not allowed in our definition of decreasing box moves for SYTs. We show that there exists a sequence of admissible decreasing box moves which transforms $\Phi^\beta_\gamma(\Gamma)$ into $\Phi^\beta_\gamma(\Delta)$.

Denote by $e_1<e_2<\ldots<e_n$ the entries of the $e$-th row of the tableau $\Phi^\beta_\gamma(\Gamma)$ and by 
$f_1<f_2<\ldots<f_m$  the entries of the $f$-th row of $\Phi^\beta_\gamma(\Gamma)$. Note that $m+1\leq n-1$. Indeed, in $\Phi^\beta_\gamma(\Delta)$ the entry $\singlebox s$ is moved from row $e$ to row $f$, and the resulting tableau is standard. Hence the $f$-th row in $\Phi^\beta_\gamma(\Delta)$ has one more entry than in $\Phi^\beta_\gamma(\Gamma)$, while the $e$-th row has one fewer.

Let $\hat{s}$ be the index such that $e_{\hat{s}}=s$.

The procedure transforming $\Phi^\beta_\gamma(\Gamma)$ into $\Phi^\beta_\gamma(\Delta)$ by admissible decreasing box moves is divided into three steps. The construction will only affect the entries in the $e$-th and $f$-th rows. We first describe the tableaux obtained after each step. Afterwards we explain how these steps are realized by admissible decreasing box moves.

In the {\bf First step}, we obtain a standard Young tableau $\Sigma_1$ in which the $e$-th row consists of the entries $e_1,e_2,\ldots,$ $e_{n-1},\max\{e_n,f_m\}$ and the $f$-th row consists of the entries $f_1,f_2,\ldots,f_{m-1},\min\{e_n,f_m\}$. If $m=0$, that is, if the $f$-th row is empty, then the first step is trivial and we set $\Sigma_1=\Phi^\beta_\gamma(\Gamma)$.

After the {\bf Second step}, we obtain a standard Young tableau $\Sigma_2$ in which the $e$-th row consists of the entries $e_1,e_2,\ldots,e_{n-1}$, and the $f$-th row consists of the entries $f_1,f_2,\ldots,f_{m},e_n$.

In the {\bf Third step}, we transform $\Sigma_2$ into the tableau $\Phi^\beta_\gamma(\Delta)$.

{\bf First step.} 
If $m=0$, that is, if the $f$-th row of $\Phi^\beta_\gamma(\Gamma)$ is empty, then the first step is trivial, we set $\Sigma_1=\Phi^\beta_\gamma(\Gamma)$ and proceed to the Second step.

If $e_n =\max\{e_n,f_m\}$, then again we set $\Sigma_1=\Phi^\beta_\gamma(\Gamma)$.

Otherwise, we choose $i$ such that $f_{i-1}<e_n<f_i$, where we set $f_0=0$.
We obtain a tableau $\Pi_0$ by swapping the entries $e_n$ and $f_i$. 
The smaller entry $e_n$ moves from row $e$ to row $f>e$, hence to a higher-numbered row, so this is a decreasing box move and $\Pi_0\boxleq \Phi^\beta_\gamma(\Gamma)$. The tableau $\Pi_0$ is again standard. Indeed, in the $e$-th row the entry $e_n$ is replaced by the larger entry $f_i$, and $e_n$ was the last entry of this row. In the $f$-th row the entry $f_i$ is replaced by $e_n$, and by the choice of $i$ we have $f_{i-1}<e_n<f_i$. Hence the increasing order in both rows is preserved. The column condition is also preserved. First, consider the position in row $e$, where $f_i$ replaces $e_n$. The entry above $e_n$, if it exists, was already smaller than $e_n$ in $\Phi^\beta_\gamma(\Gamma)$, hence it is also smaller than $f_i$. Moreover, there is no box below $e_n$, since $m+1\leq n-1$. Thus the column condition is satisfied in this position. Next, consider the position in row $f$, where $e_n$ replaces $f_i$. Let $a$ be the entry above $f_i$. We claim that $a<e_n$. Indeed, otherwise, after moving $\singlebox s$ from row $e$ to row $f$ and reordering the $f$-th row, an entry not greater than $e_n$ would appear below $a$, contradicting the fact that $\Phi^\beta_\gamma(\Delta)$ is a standard Young tableau. Hence $a<e_n$. Finally, since $e_n<f_i$, the entry $e_n$ is also smaller than the entry below $f_i$, if such an entry exists. Therefore the column condition is preserved in this position as well.

Starting from $\Pi_0$, we successively swap the entry currently occupying the last position in the $e$-th row with the entries $f_{i+1},f_{i+2},\ldots,f_m$. Equivalently, we perform the swaps
\[
(f_i,f_{i+1}),\ (f_{i+1},f_{i+2}),\ \ldots,\ (f_{m-1},f_m).
\]
At each step, the smaller entry moves from row $e$ to row $f>e$, hence each swap is a decreasing box move. By the same argument as above, each intermediate tableau is standard. These swaps move the entry $f_m$ to the last position of the $e$-th row. The resulting tableau has the entries $e_1,e_2,\ldots,e_{n-1},f_m=\max\{e_n,f_m\}$ in the $e$-th row and the entries
\(
f_1,\ldots,f_{i-1},e_n,f_i,\ldots,f_{m-1}
\)
in the $f$-th row. Denote this tableau by $\Sigma_1$. We have
\[
\Sigma_1\boxleq\Pi_0\boxleq \Phi^\beta_\gamma(\Gamma).
\]

{\bf Second step.} We obtain $\Sigma_2$ from $\Sigma_1$ by moving the last box of the $e$-th row to the last position of the $f$-th row. This is an admissible decreasing box move. Indeed, if $m=0$, then the $f$-th row of $\Sigma_1$ is empty. Otherwise, by the construction of $\Sigma_1$, the last entry of the $e$-th row is $\max\{e_n,f_m\}$, hence it is greater than the last entry of the $f$-th row. Therefore the row condition is preserved after moving this box to the end of row $f$. The column condition is preserved as well. The box removed from the $e$-th row is the last one, so no column condition is affected there. It remains to check the new position in row $f$. Let $a$ be the entry above this position, and let $x$ be the entry moved from row $e$ to row $f$ in this step. We claim that $a<x$. Indeed, otherwise, after moving the entry $\singlebox s$ from row $e$ to row $f$ and reordering the $f$-th row, an entry not greater than $x$ would appear below $a$, contradicting the fact that $\Phi^\beta_\gamma(\Delta)$ is a standard Young tableau. Hence $a<x$, and therefore the column condition is preserved in the new position. Thus $\Sigma_2$ is obtained from $\Sigma_1$ by a single decreasing box move, and hence
\(
\Sigma_2\boxleq \Sigma_1.
\)

{\bf Third step.}  
If $\hat{s}=n$, then $\Sigma_2= \Phi^\beta_\gamma(\Delta)$ and we are done.
Otherwise, $\singlebox s=e_{\hat{s}}$ is still in the $e$-th row. We proceed as in the proof of \cite[Lemma 3.1]{kst2019}.

We fix $j_1$ to be the maximal index such that $f_{j_1}<e_n$ and let $j_t<j_{t-1}<\ldots<j_1$ be all indices such that
$$f_{j_{t}-1}<e_{\hat{s}}<f_{j_t}<f_{j_{t-1}}<\ldots<f_{j_1}<e_n.$$

In the case $\hat{s}=n-1$, apply the following sequence of swaps:
$$(e_{\hat{s}}, f_{j_t} ), (f_{j_t} , f_{j_{t-1}} ), \ldots, (f_{j_{2}} , f_{j_{1}} ),(f_{j_{1}},e_n).$$ This sequence of swaps inserts the entry $s=e_{n-1}$ into the $f$-th row in the correct position (preserving the increasing order), moves the entry $e_n$ to the end of the $f$-th row, and leaves all other entries unchanged. We can show, similarly as in the previous steps, that each swap is a decreasing box move and that the standard Young tableau property is preserved.

Suppose that $\hat{s}<n-1$. We repeat the same construction. Starting from $\Sigma_2$, the $e$-th row contains
\[
e_1,\ldots,e_{n-1},
\]
while $e_n$ lies in the $f$-th row. Applying the sequence of swaps (described just above), we obtain a tableau whose $e$-th row contains
\[
e_1,\ldots,e_{n-2},e_n,
\]
while $e_{n-1}$ lies in the $f$-th row. Repeating this procedure, at the next step the $e$-th row contains
\[
e_1,\ldots,e_{n-3},e_{n-1},e_n,
\]
while $e_{n-2}$ lies in the $f$-th row, and so on. After finitely many steps, the entry $e_{\hat{s}}=\singlebox s$ lies in the $f$-th row, and the obtained tableau is precisely $\Phi^\beta_\gamma(\Delta)$. At each step, the same argument as above shows that the swaps are decreasing box moves and that all intermediate tableaux are standard.

We have shown that if $\Delta$ is obtained from $\Gamma$ by a single decreasing box move, then 
\[
\Phi^\beta_\gamma(\Delta)\boxleq \Phi^\beta_\gamma(\Gamma).
\]
Hence, by the definition of the box order, $\Delta\boxleq \Gamma$ implies 
\[
\Phi^\beta_\gamma(\Delta)\boxleq \Phi^\beta_\gamma(\Gamma).
\]

Conversely, suppose that $\Phi^\beta_\gamma(\Delta)$ is obtained from $\Phi^\beta_\gamma(\Gamma)$ by a decreasing box move. If these tableaux have the same shape, then the claim follows from \cite{kst2019}. Otherwise, the move consists in moving the last box of a row to a higher-numbered row, which corresponds to increasing an entry in the LR-tableau. Thus $\Delta$ is obtained from $\Gamma$ by a decreasing box move.

\end{enumerate}
\end{proof}

To illustrate procedure described in the proof  
consider the following LR-tableaux, which are related by the decreasing box move which increases  entry $\singlebox 1$ in row 3.
$$\Delta = \ytableausetup{centertableaux}\ytableausetup{smalltableaux}
\ytableaushort
{\none\none\none\none\none\none\none\none\none1,\none\none\none\none\none\none\none\none2, \none\none\none\none\none\none\none3,\none\none\none\none\none\none1,\none\none\none\none\none2,\none\none\none\none1,\none\none\none1,\none\none3,\none2,3}*{10,9,8,7,6,5,4,3,2,1}
\boxl 
\ytableaushort
{\none\none\none\none\none\none\none\none\none1,\none\none\none\none\none\none\none\none2, \none\none\none\none\none\none\none1,\none\none\none\none\none\none1,\none\none\none\none\none2,\none\none\none\none1,\none\none\none1,\none\none3,\none2,3}*{10,9,8,7,6,5,4,3,2,1}
=\Gamma
$$
The corresponding standard Young tableaux are
$$\Phi^\beta_\gamma(\Delta)= \ytableausetup{centertableaux}\ytableausetup{smalltableaux}
\ytableaushort
{1467,259,38{10}}*{4,3,3}
\quad\quad \Phi^\beta_\gamma(\Gamma)=
\ytableaushort
{13467,259,8{10}}*{5,3,2}
$$
The moving up the entry $\singlebox 3$ of $\Phi^\beta_\gamma(\Gamma)$ from the first row to the third row  is not a legal decreasing box move. But $\Phi^\beta_\gamma(\Delta)$ can be obtained from  $\Phi^\beta_\gamma(\Gamma)$  by the sequence of the following decreasing box moves:
\begin{enumerate}
\item swap $7$ and $8$,
\item swap $8$ and $10$, 
\item wind up box $\singlebox {10}$ to the third row,
\item swap $6$ and $7$,
\item swap $4$ and $6$
\item swap $3$ and $4$.
\end{enumerate}
The points 1. and 2. correspond to the first step described in the proof, so the tableau $\Sigma_1$ has a form:
$$\Sigma_1=\ytableaushort
{1346{10},259,78}*{5,3,2}
$$ The point 3. corresponds to the second step and the $\Sigma_2$ is the following tableau: 
$$\Sigma_2=\ytableaushort
{1346,259,78{10}}*{4,3,3}
$$
The points 4., 5., 6. correspond to the last step.

As a~consequence we get also.

\begin{theorem}\label{thm-main0}
For partitions $\beta,\gamma$ such that $\beta\setminus\gamma$ is a~rook strip and $\lvert \beta \rvert -\lvert \gamma \rvert =r$, we have:
$$ \boxleq = \domleq $$
on $\mathcal{T}^\beta_{r,\gamma}$.
\end{theorem}

\begin{proof}
Let $\Delta,\Gamma\in \mathcal{T}^\beta_{r,\gamma}$. By Lemma \ref{lem-LR-SYT-box-dom} and Theorem \ref{thm-main} we have the following equivalences:
$$\Delta\boxleq\Gamma \;\Leftrightarrow\; \Phi^\beta_\gamma(\Delta)\boxleq \Phi^\beta_\gamma(\Gamma) \;\Leftrightarrow\; \Phi^\beta_\gamma(\Delta)\domleq \Phi^\beta_\gamma(\Gamma) \;\Leftrightarrow\; \Delta \domleq\Gamma.$$
\end{proof}

%
\section*{Statements and Declarations}
This research did not receive any specific grant from funding agencies in the public, commercial, or not-for-profit sectors.

%






\bibliographystyle{elsarticle-num}
\bibliography{bibliozhabiin}







\end{document}